\documentclass[10pt]{article}
\usepackage{amsmath,amssymb,amsthm,amscd}
\numberwithin{equation}{section}

\newtheorem{prop}{Proposition}[section]
\newtheorem{theorem}[prop]{Theorem}
\newtheorem{lemma}[prop]{Lemma}
\newtheorem{corollary}[prop]{Corollary}
\newtheorem{remark}[prop]{Remark}

\newtheorem{conjecture}[prop]{Conjecture}
\newtheorem{acknowledgment}[prop]{Acknowledgment}

\def\<{\langle}
\def\>{\rangle}
\def\({\left(}
\def\){\right)}

\def\p{\partial}

\def\Ric{{\rm Ric}}

\begin{document}

\title{Globally Existing K\"ahler-Ricci Flows}

\author{the University of Sydney \\
Zhou Zhang \footnote{Partially supported by Australian Research Council Discovery Project: DP110102654.}} 
\date{}
\maketitle

\begin{abstract}

We consider the general K\"ahler-Ricci flows which exist for all time. The zeroth order control on the flow metric potential for various infinite time singularities is the focus. The possible semi-amplness for numerically effective classes serves as the main motivation. 

\end{abstract}

\section{Introduction}

$X$ is a closed K\"ahler manifold with $\dim_{\mathbb{C}}X=n\geqslant 2$. For a real smooth closed $(1, 1)$-form $L$ and a smooth volume form $\Omega$ over $X$, we consider the following geometric evolution equation: 
\begin{equation}
\label{eq:general-flow}
\frac{\p\widetilde\omega (t)}{\p t}=-\Ric\(\widetilde\omega(t)\)-\widetilde\omega(t)+\Ric(\Omega)+L, ~~~~\widetilde\omega(0)=\omega_0,  
\end{equation}
where $\Ric(\Omega):=-\sqrt{-1}\p\bar\p\log\bigl(\frac{\Omega}{V_{E}}\bigr)$ with $V_E$ being the Euclidean volume form with respect to a local holomorphic coordinate chart $\{z_1, \cdots, z_n\}$, as a natural generalization of the Ricci form for K\"ahler metric. This evolution equation was first considered in \cite{tsu-2} and then studied in \cite{t-znote}. It's easy to see that this flow is parabolic, just as the more classic K\"ahler-Ricci flow. In fact, when $L=-\Ric(\Omega)$, it becomes the following standard K\"ahler-Ricci flow over $X$, 
\begin{equation}
\label{eq:krf}
\frac{\p\widetilde\omega (t)}{\p t}=-\Ric\(\widetilde\omega(t)\)-\widetilde\omega(t), ~~~~\widetilde\omega(0)=\omega_0. 
\end{equation}
The flow (\ref{eq:krf}) is studied extensively in the past decade because of the profound relation with the Minimal Model Program, as proposed by Gang Tian in \cite{tian02}. Of course, the flow (\ref{eq:general-flow}) is more flexible and has great advantages in applications, as illustrated below.  

Considering (\ref{eq:general-flow}) in the cohomology space $H^{1, 1}(X;\mathbb{R}):=H^2(X; \mathbb{R})\cap H^{1,1}(X; \mathbb{C})$, we know the solution $\widetilde\omega_t$ satisfies $[\widetilde\omega_t]=[L]+e^{-t}([\omega_0]-[L])$. Thus if we set $\omega_t=L+e^{-t}(\omega_0-L)$, then by $\p\bar\p$-Lemma, $\widetilde\omega(t)=\omega_t+\sqrt{-1}\p\bar\p u$ with $u$ satisfying 
\begin{equation}
\label{eq:skrf}
\frac{\p u}{\p t}=\log\frac{(\omega_t+\sqrt{-1}\p\bar\p u)^n}{\Omega}-u, ~~~~u(\cdot, 0)=0.
\end{equation}

By the optimal existence result in \cite{t-znote}, the smooth metric solution exists as long as $[\omega_t]=[L]+e^{-t}([\omega_0]-[L])$ stays in the K\"ahler cone of $X$ which is an open convex cone sitting in $H^{1, 1}(X; \mathbb{R})$. So if we define the time of singularity,  
$$T:=\sup\{t\,|\,[\omega_t]~\text{K\"ahler}\}\in [0, \infty),$$
then the flow (\ref{eq:general-flow}) has the smooth metric solution existing in $[0, T)$. 

In this work, we focus on the case of $T=\infty$. In this case, $[L]$ is in the closure of the K\"ahler cone of $X$, i.e., being numerically effective (nef. for simplicity). If $[L]$ is K\"ahler, then by the non-degenerate version of the result in \cite{t-znote} or more explicitly \cite{thesis}, we know that for any initial K\"ahler metric $\omega_0$, the flow converges smoothly to the unique solution of $$-\Ric\(\widetilde\omega\)-\widetilde\omega+\Ric(\Omega)+L=0,$$ 
greatly extending the earlier study by H. D. Cao on the classic K\"ahler-Ricci flow (\ref{eq:krf}) for $[\omega_0]=-c_1(X)$ being K\"ahler.  

Now we focus on the case that $[L]$ is on the boundary of the K\"ahler cone (nef. but not K\"ahler). Then the flow must develop infinite time singularities since it is impossible to converge smoothly to a K\"ahler metric.  

In the case of $L=-\Ric(\Omega)$, i.e., for the standard K\"ahler-Ricci flow ({\ref{eq:krf}}), by the famous Abundance Conjecture in algebraic geometry, $[L]=K_X$ is expected to be semi-ample, i.e., base-point-free. Then for some integer $m$ large enough, (the global holomorphic sections of) $mK_X$ would generate a holomorphic map $F: X\to\mathbb{CP}^N$ where $N+1=\dim_{\mathbb{C}}H^0(X, mK_X)$. In general, the image is a (singular) variety and clearly the dimension $\dim_{\mathbb{C}}F(X)\leqslant n$, which is the Kodaira dimension of $X$. If $K_X$ also big (equivalent to $K^n_X>0$), then it's well known by algebraic geometry argument (for example, in \cite{kawa}) that $K_X$ is semi-ample. There have been very recent works \cite{gsw} and \cite{tian-zzl} regarding the global geometric behaviour of the flow in this case. Anyway, for this conjecture it is left to deal with the (collapsed) case, i.e., $K^n_X=0$.    

Meanwhile, the situation can be a lot more complicated if one considers a general nef. class $[L]$. In fact, as described in \cite{begz-big}, there is the famous example by Serre about a nef. and big integral class which is not semi-ample. This more general flow (\ref{eq:general-flow} provides a natural way to study any nef. class, which is the original motivation in \cite{tsu-2}. From the differential geometry point of view, it amounts to the analysis of infinite time singularities, and we focus on the zeroth order behaviour in this work. 

\vspace{0.1in}

Let's begin by pointing out that, in principle, infinite time singularities are quite different from finite time singularities, with $[L]$ being on the boundary of the K\"ahler cone in the former case and on the complement of the closure of the K\"ahler cone in the latter one. However, motivated by \cite{tian-survey}, we have the following result in \cite{weak-limit}.  
\begin{prop}
\label{prop:non-collapsed}
Consider the K\"ahler-Ricci flow (\ref{eq:general-flow}) and (\ref{eq:skrf}). For the time of singularity $T\in (0, \infty]$, if $[\omega_T]^n>0$, then we have $\widetilde\omega(T)=\omega_T+\sqrt{-1}\p\bar\p u(T)$ with $u(T)\in PSH_{\omega_T}(X)$ and $\widetilde\omega(t)\to\widetilde\omega(T)$ in the weak sense as $t\to T$.
\end{prop}

Indeed, for the finite time singularity case, the general existence of a flow weak limit is proved in \cite{weak-limit}. Meanwhile, for the infinite time singularity case, we show here that the converse is also true as in the following theorem.  

\begin{theorem}
\label{thm:weak-limit-vol-non-collapsed}
Consider the K\"ahler-Ricci flow (\ref{eq:general-flow}) and (\ref{eq:skrf}). If the time of singularity $T=\infty$ and $u(\cdot, t)\to u(\cdot, \infty)$ weakly as $t\to\infty$ with $u(\cdot, \infty)\in PSH_{\omega_\infty}(X)$, then the global volume is non-collapsed, i.e., $[L]^n>0$.
\end{theorem}

In the global volume collapsed case, we have the following result on the lower control of pointwise volume collapsing rate. 

\begin{theorem}
\label{thm:collapsed-lower}
Consider the K\"ahler-Ricci flow (\ref{eq:general-flow}) and (\ref{eq:skrf}). If the time of singularity $T=\infty$, then for any $\varphi\in PSH_{L}(X)$, there is $C$ such that 
$$\frac{\p u}{\p t}+u+nt\geqslant \varphi+C.$$
\end{theorem}

Clearly it is equivalent in Theorem \ref{thm:collapsed-lower} to take $\varphi$ of minimal singularities. 

\vspace{0.1in}

\noindent{\bf Notations:} $C$ stands for a positive constant, possibly different at places. $f(t)\thicksim g(t)$ means $\lim_{t\to \infty}\frac{f(t)}{g(t)}=C$.

\section{Infinite Time Singularities}

In the following, the Laplacian $\Delta$ is always with respect to the metric along the flow, $\widetilde\omega(t)$. Applying Maximum Principle directly for (\ref{eq:skrf}), we have 
$$u\leqslant C.$$

The $t$-derivative of (\ref{eq:skrf}) is   
\begin{equation}
\label{eq:t-derivative}
\frac{\partial}{\partial t}\(\frac{\partial u}{\partial t}\)=\Delta\(\frac{\partial u}{\partial t}\)-e^{-t}\<\widetilde\omega(t), \omega_0-L\>-\frac{\partial u}{\partial t}.
\end{equation}
where $\<\omega, \alpha\>$ stands for the trace of the smooth real closed $(1, 1)$-form $\alpha$ with respect to the K\"ahler metric $\omega$. Equivalently, $\<\omega, \alpha\>=(\omega, \alpha)_\omega$ where $(\cdot, \cdot)_\omega$ is the hermitian inner product with respect to $\omega$. 

Taking $t$-derivative for (\ref{eq:t-derivative}), we get 
\begin{equation}
\label{eq:2-t-derivative}
\frac{\partial}{\partial t}\(\frac{\partial^2 u}{\partial t^2}\)=\Delta\(\frac{\partial^2 u}
{\partial t^2}\)+e^{-t}\<\widetilde\omega(t),\omega_0-L\>-\frac{\partial^2 u}{\partial t^2}-
{\vline\frac{\partial\widetilde\omega(t)}{\partial t}\vline}^2_{\,\widetilde\omega(t)}.
\end{equation}

Sum up (\ref{eq:t-derivative}) and (\ref{eq:2-t-derivative}) to get
$$\frac{\partial}{\partial t}\(\frac{\partial^2 u}{\partial t^2}+\frac{\partial u}{\partial t}\)
=\Delta\(\frac{\partial^2 u}{\partial t^2}+\frac{\partial u}{\partial t}\)-\(\frac{\partial^2 u}
{\partial t^2}+\frac{\partial u}{\partial t}\)-{\vline\frac{\partial\widetilde\omega(t)}{\partial t}\vline}^2_{\,\widetilde\omega(t)}.$$
Standard Maximum Principle argument then gives
\begin{equation}
\label{ineq:main} 
\frac{\partial^2 u}{\partial t^2}+\frac{\partial u}{\partial t}\leqslant Ce^{-t},
\end{equation} 
which implies the essential decreasing of volume form along the flow, i.e. , 
$$\frac{\partial}{\partial t}\(\frac{\partial u}{\partial t}+u+Ce^{-t}\)\leqslant 0.$$

(\ref{ineq:main}) also tells us  
$$\frac{\partial u}{\partial t}\leqslant Ce^{-t}+Cte^{-t}\leqslant Ce^{-\frac{t}{2}}.$$ 
So $u+Ce^{-\frac{t}{2}}$ is decreasing along the flow. By the basic property of plurisubharmonic functions, as long as this term (or equivalently $u$) doesn't converge to $-\infty$ uniformly as $t\to \infty$, $u$ would converge to some $u(\cdot, \infty)\in PSH_{\omega_\infty}(X)$ and $L+\sqrt{-1}\p\bar\p u(\cdot, \infty)$ is the flow limit for $\widetilde\omega(t)$ as $t\to \infty$ in the sense of current. By all means, we can define
$$v(\cdot):=\lim_{t\to\infty}\(\frac{\p u}{\p t}+u\)(\cdot, t),~\text{and}$$
$$u(\cdot, \infty):=\lim_{t\to\infty}u(\cdot, t)\in PSH_{\omega_\infty}(X)\cup\{\text{virtual constant function}-\infty\},$$
where $v(\cdot)$ is just a (generalized) function defined over $X$ valued in $[-\infty, \infty)$ which is clearly bounded from above and (essentially) upper semi-continuous. 

\vspace{0.1in}

\noindent{\bf Notation:} for $a$ and $b\in [-\infty, \infty)$, $a\leqslant b$ and $a=b$ are undertood in the natural way. 

\vspace{0.1in}

The following lemma is important for our purpose. 

\begin{lemma}
\label{lem:key}
Using the above notations, $u(\cdot, \infty)=v(\cdot)$. 
\end{lemma}

\begin{proof}
We already have the uniform upper bound of $\frac{\p u}{\p t}$. Thus $\frac{\p u}{\p t}+u\leqslant u+C$, and so $v(\cdot)\leqslant u(\cdot, \infty)+C$. Hence if $u(x, \infty)=-\infty$ at $x\in X$, so is $v(x)$, i.e., 
$$\{x\in X\,\vline ~u(x, \infty)=-\infty\}\subset\{x\in X\,\vline ~v(x)=-\infty\}.$$ 

On the other hand, if $v(x)=-\infty$ and $u(x, \infty)\in\mathbb{R}$, then $\frac{\p u}{\p t}(x, t)\to -\infty$ as $t\to\infty$, contradicting $u(x, \infty)\in\mathbb{R}$ by the Fundamental Theorem of Calculus. So $v(x)=-\infty$ also implies $u(x, \infty)=-\infty$, and 
$$\{x\in X\,\vline ~v(x)=-\infty\}\subset\{x\in X\,\vline ~u(x, \infty)=-\infty\}.$$ 

Thus we conclude  
$$\{x\in X\,\vline ~v(x)=-\infty\}=\{x\in X\,\vline ~u(x, \infty)=-\infty\}=:U\subset X,$$ 
$$\{x\in X\,\vline ~v(x)\in \mathbb{R}\}=\{x\in X\,\vline ~u(x, \infty)\in\mathbb{R}\}=U^c\subset X.$$ 

At any $x\in U^c$, we clearly have $\frac{\p u}{\p t}(x, t)\to v(x)-u(x, \infty)$ as $t\to\infty$, which has to be $0$ by the existence of $u(x, \infty)\in\mathbb{R}$. So we conclude that
$$u(\cdot, \infty)=v(\cdot)$$
as generalized function on $X$ valued in $\mathbb{R}\cup\{-\infty\}=[-\infty, \infty)$. 

\end{proof}

\subsection{The Non-collapsed Case: The Proof of Theorem \ref{thm:weak-limit-vol-non-collapsed}}

With Lemma \ref{lem:key}, the proof of Theorem \ref{thm:weak-limit-vol-non-collapsed} becomes straightforward. 

If $u(\cdot, \infty)\in PSH_{\omega_\infty}(X)$, and so not identically $-\infty$, then neither is $v$. So we have
$$[L]^n=\lim_{t\to\infty}[\omega_t]^n=\lim_{t\to\infty}[\widetilde\omega_t]^n=\lim_{t\to\infty} \int_X e^{\frac{\p u}{\p t}+u}\Omega=\int_X e^v\Omega=\int_X e^{u(\cdot, \infty)}\Omega>0.$$ 

\begin{remark}

Consider (\ref{eq:general-flow}) with $[L]$ semi-ample and big. By the discussion in \cite{t-znote}, \cite{zzo} and \cite{bound-r}, we know that $u(\cdot, \infty)\in C^0(X)$. So as $t\to\infty$, the convergence of $u$ is uniform, and also is the convergence of $\frac{\p u}{\p t}$ to $0$. For (\ref{eq:krf}), this is the case for $K_X$ nef. and big.  

\end{remark}
 
\subsection{The Collapsed Case}

We now focus on the infinite time singularities with collapsed global volume, i.e., $[L]^n=0$. For them, the following are already known.  
 
\begin{enumerate}

\item For (\ref{eq:skrf}), $u(\cdot, t)\to -\infty$ and $\frac{\p u}{\p t}\to 0$ uniformly as $t\to\infty$ over $X$ by  Theorem {\ref{thm:weak-limit-vol-non-collapsed}}.  

\item $[\omega_t]^n\thicksim e^{-kt}$ for some integer $k\in\{1, 2, \cdots, n\}$. 

\end{enumerate}

The goal here is to control the rate in which $u$ and $\frac{\p u}{\p t}+u$ approaches $-\infty$. 

\vspace{0.1in}

Regarding the upper control, we have the following simple estimation. For simplicity of notation, assuming $\int_X\Omega=1$,  
\begin{equation}
e^{\int_X (\frac{\p u}{\p t}+u)\Omega}\leqslant \int_X e^{ \frac{\p u}{\p t}+u}\Omega=\int_X \widetilde\omega(t)^n=[\omega_t]^n\leqslant Ce^{-kt},
\end{equation}
and so $\int_X (\frac{\p u}{\p t}+u)\Omega\leqslant -kt+C$. 

\vspace{0.1in}

Now we recycle some argument in \cite{weak-limit} to gain some lower control. Start by transforming (\ref{eq:t-derivative}) into the following form,
\begin{equation}
\label{eq:2} 
\frac{\partial}{\partial t}\(\frac{\partial u}{\partial t}+u\)=\Delta\(\frac{\partial u}{\partial t}+u\)-n+\<\widetilde\omega(t), L\>.
\end{equation}

As $[L]$ is on the boundary of the K\"ahler cone for $X$, there exists $\varphi\in PSH_{L}(X)$. For this fact, one could make use of the sequential limit construction in \cite{tian-survey} as mentioned in \cite{weak-limit}. We prove the following theorem which is Theorem \ref{thm:collapsed-lower},    
\begin{theorem}
For (\ref{eq:skrf}), if the singularity time $T=\infty$, then for any $\varphi\in PSH_{L}(X)$, there is $C$ such that 
$$\frac{\p u}{\p t}+u+nt\geqslant \varphi+C.$$
\end{theorem}

To begin with, the fundamental regularization result by Demailly (Theorem 1.6 in \cite{demailly-2014}) provides us with a decreasing approximation sequence for $\varphi$, $\{\varphi_m\}^\infty_{m=1}$ which satisfies:
\begin{itemize}

\item $\varphi_m\in PSH_{L+\frac{1}{m}\omega_0}(X)$;

\item $\varphi_m\in C^\infty(X\setminus Z_m)$ with $Z_m\subset Z_{m+1}$ being analytic subvarieties of $X$. Furthermore, $\varphi_m$ has logarithmic poles along $Z_m$, i.e., locally being the logarithm of the sum of squares for finitely many holomorphic functions vanishing along $Z_m$.

\end{itemize} 

The following equation is obtained by simple manipulation of (\ref{eq:t-derivative})
\begin{equation}
\label{eq:3}
\frac{\partial}{\partial t}\((1-e^t)\frac{\partial u}{\partial t}+u\)=\Delta\((1-e^t)\frac{\partial u}{\partial t}+u\)-n+\<\widetilde\omega(t), \omega_0\>.
\end{equation}

Then we have the following combination of (\ref{eq:2}) and (\ref{eq:3}),
\begin{equation} 
\begin{split}
&~~ \frac{\partial}{\partial t}\(\frac{1}{m}[(1-e^t)\frac{\partial u}{\partial t}+u]+[\frac{\partial u}{\partial t}+u]\) \\
&= \Delta\(\frac{1}{m}[(1-e^t)\frac{\partial u}{\partial t}+u]+[\frac{\partial u}{\partial t}+u]\)-\frac{n(1+m)}{m}+\<\widetilde\omega(t), \frac{1}{m}\omega_0+L\>.
\end{split} \nonumber
\end{equation}

Using $\varphi_m$, we modify the above equation as follows: 
\begin{equation} 
\label{eq:5}
\begin{split}
&~~ \frac{\partial}{\partial t}\(\frac{1}{m}[(1-e^t)\frac{\partial u}{\partial t}+u]+[\frac{\partial u}{\partial t}+u]-\varphi_m+\frac{n(1+m)t}{m}\) \\
&= \Delta\(\frac{1}{m}[(1-e^t)\frac{\partial u}{\partial t}+u]+[\frac{\partial u}{\partial t}+u]-\varphi_m+\frac{n(1+m)t}{m}\) \\
&~~~~ +\<\widetilde\omega(t), \frac{1}{m}\omega_0+L+\sqrt{-1}\p\bar\p\varphi_m\>, 
\end{split}
\end{equation}
where $ \frac{1}{m}\omega_0+L+\sqrt{-1}\p\bar\p\varphi_m$ is smooth and positive over $X\setminus Z_m$. Since $\varphi_m\in C^\infty(X\setminus Z_m)$ and has $-\infty$ poles along $Z_m$, the spatial minimum of 
$$\frac{1}{m}[(1-e^t)\frac{\partial u}{\partial t}+u]+[\frac{\partial u}{\partial t}+u]-\varphi_m+\frac{n(1+m)t}{m}$$ is always achieved in $X\setminus Z_m$, where everything is smooth. We then apply the standard Maximum Principle argument to conclude that its global minumum has to be obtained at the initial time, and so 
$$\frac{1}{m}[(1-e^t)\frac{\partial u}{\partial t}+u]+[\frac{\partial u}{\partial t}+u]-\varphi_m+\frac{n(1+m)t}{m}\geqslant -C,$$
which is uniform for all $m$'s over $X\times [0, \infty)$. Here, we make use of the uniform upper bound of $\varphi_m$'s. For any space-time point, take $m\to \infty$ and arrive at  
$$\frac{\partial u}{\partial t}+u-\varphi+nt\geqslant -C$$
over $X\times [0, \infty)$. The theorem is proved. 

\begin{remark}

Combining with the known upper bounds for $u$ and $\frac{\p u}{\p t}$, we also have the lower control, $-nt+\varphi-C$, for both $u$ and $\frac{\p u}{\p t}$. This $\varphi$ can certainly been chosen to be of minimal singularities, as also in \cite{weak-limit}. 

\end{remark}

A direct consequence of Theorem \ref{thm:collapsed-lower} is the following corollary. 
\begin{corollary}
\label{cor:derivative-seq-lower}
In the setting of Theorem \ref{thm:collapsed-lower}, for any $\epsilon>0$, there is a time sequence, $\{t_i\}_{i=1}^\infty$, approaching $\infty$ such that 
$$\max_{X\times\{t_i\}}\frac{\p u}{\p t}\geqslant -n-\epsilon.$$  
\end{corollary}

The proof is simple by contradiction since otherwise $u$ will have an upper bound like $-(n+\epsilon)t+C$.   

In the finite time singularity collapsed case, $\frac{\p u}{\p t}$ tends to $-\infty$ uniformly, very different from the situation here. 

\section{Further Remarks}

There is a lot left to be sorted out regarding the behaviour of the flow (\ref{eq:general-flow}), especially for the collapsed case. We now describe the following special case, also mentioned in \cite{weak-limit}, which can be viewed as the main motivation of this study. 

\vspace{0.1in}

Suppose $[L]$ is semi-ample, and so $[L]$ gives a fibration structure of $X$ with general fibre of dimension $0<k\leqslant n$, i.e., $F: X\to\mathbb{CP}^N$ with $mL=F^*[\omega_{_{FS}}]$ and $F(X)$ of dimension $n-k$. 

In this case, $u\thicksim -kt$, which can be seen as follows. Begin with the following scalar potential flow
$$\frac{\p v}{\p t}=\log\frac{(\omega_t+\sqrt{-1}\p\bar\p v)^n}{\Omega}-v+kt, ~~~~v(\cdot, 0)=0.$$
Clearly, it still corresponds to the same metric flow (\ref{eq:general-flow}) and the relation between $u$ and $v$ is 
$$u=v+f(t)~~\text{with}~~\frac{d f}{d t}+f=-kt, ~~f(0)=0.$$
It's easy to get $f(t)\thicksim -kt$ and $\frac{d f}{d t}\thicksim -k$. Rewrite the equation of $v$ as follows
$$(\omega_t+\sqrt{-1}\p\bar\p v)^n=e^{-kt}e^{\frac{\p v}{\p t}+v}\Omega$$
and apply the $L^\infty$ estimates in \cite{egz2} and \cite{demailly-pali}, we have $|v|\leqslant C$ uniformly for all time. Hence $u\thicksim -kt$, tending to $-\infty$ uniformly as $t\to\infty$. 

In fact, by the result in \cite{song-tian-scalar}, we know $\vline\frac{\p u}{\p t}\vline\leqslant C$ and so $\vline\frac{\p v}{\p t}\vline\leqslant C$. In \cite{song-tian} and \cite{fong-z}, there is more discussion about the geometry of the flow in this case.

In the case of $L=-\Ric(\Omega)$, predicted by Abundance Conjecture, the flow (\ref{eq:krf}) with infinite time singularities always behaves like this. In general, the situation can be more complicated as indicated by Serre's example, i.e., nef. (and big) might not be semi-ample. Thus, it's meaningful to search for a precise understanding of the lower order controls for $u$ and $\frac{\p u}{\p t}$ without making any assumption on the semi-amplness. Hopefully, it then provides indications on the possible semi-amplness of class $[L]$. In other words, this consideration provides a differential geometry approach of using geometric flows to tackle the Abundance Conjecture. We have the following conjecture in this direction, based on the understanding of the special case above.
 
\begin{conjecture}

Consider the K\"ahler-Ricci flow (\ref{eq:general-flow}) with infinite time singularities (i.e., $[L]$ nef. but not K\"ahler). We have, uniformly over $X\times [0, \infty)$,
$$\frac{\p u}{\p t}\geqslant -C.$$
Suppose $k\in \{0, \cdots, n\}$ is the smallest integer such that $[L]^{n-k}\cdot [\omega_0]^k\neq 0$, then for any $\varphi\in PSH_L(X)$, there is $C$ such that 
$$u\geqslant -kt-C+\varphi.$$
Furthermore, $u+kt\to \Phi$ in some proper sense with $\Phi\in PSH_L(X)$ of minimal singularities and $L+\sqrt{-1}\p\bar\p \Phi$ as the generalized K\"ahler-Einstein current associated with $[L]$.   
  
\end{conjecture}

We finish by providing some discussion of the K\"ahler-Ricci flow (\ref{eq:general-flow}) with extra differential geometry assumptions.  

\begin{itemize}  

\item $\Ric(\widetilde\omega_t)\geqslant \alpha$ for a real smooth $(1, 1)$-form $\alpha$ over $X$. 

Directly by the K\"ahler-Ricci flow equation (\ref{eq:general-flow}), we have $\widetilde\omega_t\leqslant C\omega_0$ uniformly for all time. Then it is clear that $\pm u(\cdot, t)\in PSH_{C\omega_0}(X)$ uniformly for all time. The standard argument using Green's function gives 
$${\rm osc}(u)_{X\times\{t\}}\leqslant C$$
uniformly for all time. 

\item $\Ric(\widetilde\omega_t)\leqslant \alpha$ for a smooth $(1, 1)$-form $\alpha$ over $X$.

Still by the K\"ahler-Ricci flow equation (\ref{eq:krf}), we have $\frac{\p u}{\p t}+u\in PSH_{C\omega_0}(X)$. For simplicity, we use "$\int$" to denote the average over $X$ with respect to the volume form $\omega^n_0$. Then the standard argument using Green's function kicks off the following estimation.
\begin{equation} 
\begin{split}
\frac{\p u}{\p t}+u
&\leqslant C+\int\frac{\p u}{\p t}+u \\
&\leqslant C+\log\int e^{\frac{\p u}{\p t}+u} \\
&\leqslant C+\log\int_X e^{\frac{\p u}{\p t}+u}\Omega \\
&\leqslant C+\log[\omega_t]^n\leqslant C-kt. \nonumber
\end{split}
\end{equation}

From this ordinary differential inequality, we easily deduce that $u\leqslant C-kt$. In light of $\int_X e^{\frac{\p u}{\p t}+u}\Omega=[\omega_t]^n\thicksim e^{-kt}$, we conclude   
$$\max_{X\times\{t\}}\(\frac{\p u}{\p t}\)\geqslant -C,$$
which is stronger than the conclusion of Corollary \ref{cor:derivative-seq-lower} which is for the general case.  

\end{itemize}

With these assumptions, the controls we have on $u$ and $\frac{\p u}{\p t}$ are still weaker comparing with the model case of $[L]$ being semi-ample. 

\begin{acknowledgment}

The author would like to thank Professor Gang Tian for introducing him to this wonderful field of research and constant encouragement. The work was partially carried out during the visit to the School of Mathematical Sciences at Peking University and Beijing International Center for Mathematical Research, as part of the sabbatical leave from the University of Sydney. The author would like to thank Professor Xiaohua Zhu and all institutes for this wonderful opportunity.

\end{acknowledgment}

{\it 
\noindent Zhou Zhang \\ 
\noindent Address: Carslaw Building F07, the School of Mathematics and Statistics \\
the University of Sydney, NSW 2006, Australia \\
\noindent Email: zhangou@maths.usyd.edu.au \\
\noindent Fax: + 61 2 9351 4534}

\end{document}